\documentclass[10pt,a4paper]{article}
\usepackage[latin1]{inputenc}
\usepackage{amsmath, url, hyperref}
\usepackage{amsthm, bm}
\usepackage{amsfonts}
\usepackage{amssymb}
\usepackage{graphicx}
\usepackage[margin=2.7cm]{geometry}
\usepackage{xcolor}
\newtheorem{theorem}{Theorem}
\newtheorem{lemma}{Lemma}

\author{Mahadi Ddamulira$ ^{1,2} $}
\title{Tribonacci-Lucas numbers that are palindromic concatenations of two distinct repdigits}
\date{}
\begin{document}
\maketitle

\abstract{
\noindent The Tribonacci-Lucas sequence $\{S_n\}_{n\ge 0}$ is defined by the linear recurrence relation $S_{n+3} = S_{n+2} + S_{n+1} + S_n$, for $ n\ge 0 $,  with  the initial conditions $S_0 =S_2= 3$ and $S_1 = 1$. A palindromic number is a number that remains the same when its digits are reversed.  This paper uses Baker's theory for nozero lower bounds for linear forms in logarithms of algebraic numbers, and reduction methods involving the theory of continued fraction to determine all Tribonacci-Lucas numbers that are palindromic concatenations of two distinct repdigits.}

 

\section{Introduction.}

\noindent Let $\{S_n\}_{n\ge 0}$ be the sequence of Tribonacci-Lucas numbers, defined by the linear recurrence relation  $S_{n+3} = S_{n+2} + S_{n+1} + S_n$, for $n \ge 0$, with the initial conditions  $S_0 =S_2= 3$ and $S_1 = 1$. This sequence A001644 on The On-Line Encyclopedia of Integer Sequences (OEIS) \cite{OEIS}.  The first few terms of this sequence are given by
\begin{equation*}
(S_n)_{n \geq 0} = 3, 1, 3, 7, 11, 21, 39, 71, 131, 241, 443, 815, 1499, 2757, 5071, 9327, 17155, 31553, \ldots.
\end{equation*}
\noindent The sequence $ \{S_n\}_{n\ge 0} $ is a well--known companion sequence of the classical {\it Tribonacci sequence}, $ \{T_n\}_{n\ge 0} $. That is, the two sequences satisfy the same linear recurrence relation, but with different initial conditions.
A palindromic number (or {\it palindrome}) is a number that remains the same when its digits are reversed (such as $17571$). In other words, a palindrome has reflectional symmetry across a vertical axis. A natural followup of the results in \cite{MD1, MD2} would be a characterization of \textit{palindromic} Tribonacci-Lucas numbers. That is, we study the Diophantine equation

\begin{equation}\label{eq1}
S_n = \overline{\underbrace{d_1\cdots d_1}_\text{$\ell$ times} \underbrace{d_2\cdots d_2}_\text{$m$ times} \underbrace{d_1\cdots d_1}_\text{$\ell$ times}}, 
\end{equation}
in non-negative integers $ (n,d_1,d_2,\ell, m )$, where $ d_1,d_2 \in \{0, \ldots, 9  \}, \quad d_1 > 0$, and $ \ell,m\ge 1 $. 

\noindent The main result of this paper is the following.

\begin{theorem}\label{thm1x}
The only Tribonacci-Lucas number that is a palindromic concatenation of two distinct repdigits is,  $ S_8=131 $.
\end{theorem}
\noindent
This result gives a continuation of the results presented  in \cite{MD, MD3}. The method of proof of our main results makes use of Baker's theory for nozero lower bounds for linear forms in logarithms of algebraic numbers, and reduction methods involving the theory of continued fractions. All computations are done with the aid of a simple computer program in {\tt SageMath}.

\section{Preliminary results.}

\noindent In this section we recall some facts about the Tribonacci-Lucas numbers and other preliminary lemmas, including Baker's theory for linear forms in three logs, Baker-Davenport reduction procedure, and the LLL algorithm.  These results are crucial to the proof of the main result.

\subsection{Some properties of the Tribonacci-Lucas sequence.}
Some of the important properties of the Tribonacci-Lucas sequence are stated here. It is well--known that the characteristic equation,
\begin{align*}
 \Psi(x):=x^3-x^2-x-1=0,
\end{align*}
has roots $ \alpha, ~\beta,~ \gamma = \bar{\beta} $, where
\begin{eqnarray}\label{Trib1}
\alpha =\dfrac{1+(w_1+w_2)}{3}, \qquad \beta = \dfrac{2-(w_1+w_2)+\sqrt{-3}(w_1-w_2)}{6},
\end{eqnarray}
with
\begin{eqnarray}\label{Trib2}
w_1=\sqrt[3]{19+3\sqrt{33}} \quad \text{and}\quad w_2=\sqrt[3]{19-3\sqrt{33}}.
\end{eqnarray}
The Binet-like formula for the general terms of the Tribonacci-Lucas sequence is given by,
\begin{eqnarray}\label{Trib3}
S_n = \alpha^{n}+\beta^{n}+\gamma^{n} \qquad \text{ for all} \quad n\ge 0.
\end{eqnarray}
\noindent 
Numerically, the following estimates hold:
\begin{equation}\label{Trib4}
\begin{aligned}
1.83&<\alpha<1.84;\\
0.73 < |\beta|&=|\gamma|=\alpha^{-\frac{1}{2}}< 0.74.
\end{aligned}
\end{equation}
\noindent
From \eqref{Trib1}, \eqref{Trib2}, and \eqref{Trib4}, it is easy to see that the contribution the complex conjugate roots $ \beta $ and $ \gamma $, to the right-hand side of \eqref{Trib3}, is very small. In particular, we set
\begin{equation}\label{Trib5}
\xi(n) := S_n -  \alpha^n =  \beta^n + \gamma^n. \quad \text{Then,} \quad |\xi(n)| \le  \frac{2}{\alpha^{n/2}} \quad \text{for all} \quad n \geq 1.
\end{equation}
The last inequality in \eqref{Trib5} follows from the fact that $ |\beta|=|\gamma|=\alpha^{-\frac{1}{2}} $. That is, for any $ n\ge 1 $,
\begin{align*}
|\xi (n)|=\left|\beta^{n}+\gamma^{n}\right| \le |\beta|^{n}+ |\gamma|^{n}= \alpha^{-\frac{n}{2}}+ \alpha^{-\frac{n}{2}} = 2 \alpha^{-\frac{n}{2}} = \dfrac{2}{\alpha^{n/2}}.
\end{align*}

\noindent The following estimate also holds:

\begin{lemma}\label{lem1}
Let $m \ge 1$ be a positive integer. Then

\begin{equation*}
\alpha^{m-1} \leq S_m < \alpha^{m+1}.
\end{equation*}
\end{lemma}
\begin{proof}
 Lemma \ref{lem1} follows from a simple inductive argument, with the fact that $ \alpha^3=\alpha^2+\alpha+1 $.
\end{proof}

\noindent
Let $ \mathbb{K}:=\mathbb{Q}(\alpha, \beta) $ be the splitting field of the polynomial $ \Psi $ over $ \mathbb{Q} $. Then, $ [\mathbb{K}, \mathbb{Q}]=6 $. Furthermore, $ [\mathbb{Q}(\alpha):\mathbb{Q}]=3 $. The Galois group of $ \mathbb{K} $ over $ \mathbb{Q} $ is given by
\begin{align*}
\mathcal{G}:=\text{Gal}(\mathbb{K/Q})\cong \{(1), (\alpha\beta),(\alpha\gamma), (\beta\gamma), (\alpha\beta\gamma), (\alpha\gamma\beta)\} \cong S_3.
\end{align*}
Thus, we identify the automorphisms of $ \mathcal{G} $ with the permutations of the zeros of the polynomial $ \Psi $. For example, the permutation $ (\alpha\beta) $ corresponds to the automorphism $ \sigma: \alpha \to \beta, ~\beta \to \alpha, ~\gamma \to \gamma $. This will be used later to obtain a contradiction on the size of the absolute values of certain bounds.

\subsection{Linear forms in logarithms.}
Our approach follows the standard procedure of obtaining nonzero lower bounds for certain linear forms in logarithms of algebraic numbers. The upper bounds are obtained via a manipulation of the associated Binet's formula for the given sequence. For the lower bounds, we need the celebrated Baker's theorem on linear forms in logarithms of algebraic numbers. Before stating the result, we need the definition of the (logarithmic) Weil height of an algebraic number.
\medskip

\noindent Let $\eta$ be an algebraic number of degree $d$ with minimal polynomial
\begin{align*}
P(x) = a_0 \prod_{j = 1}^d (x - \eta_j),
\end{align*}
where the leading coefficient $a_0$ is positive and the $\eta_j$'s are the conjugates of $\eta$. The logarithmic height of $\eta$ is given by
\begin{align*}
h(\eta) := \frac{1}{d} \left( \log a_0 + \sum_{j = 1}^d \log \left( \max \{|\eta_j|, 1  \}  \right)   \right).
\end{align*}
Note that, if $\eta = \frac{p}{q} \in \mathbb{Q}$ is a reduced rational number with $q > 0$, then the above definition reduces to $$h(\eta) = \log \max \{ |p|,q \}.$$
We list some well known properties of the height function below, which we shall subsequently use without reference:
\begin{equation}\label{heights}
\begin{aligned}
h(\eta_1 \pm \eta_2) & \leq h(\eta_1) + h(\eta_2) + \log 2, \\
h(\eta_1 \eta_2 ^{\pm}) & \leq h(\eta_1) + h(\eta_2), \\
h(\eta^s) & = |s| h(\eta), \quad (s \in \mathbb{Z}).
\end{aligned}
\end{equation}
Here we present the version of Baker's theorem proved by Bugeaud, Mignotte and Siksek (\cite{BMS}, Theorem 9.4), which is itself due to Matveev.
\begin{theorem}[Bugeaud, Mignotte, Siksek, \cite{BMS}]\label{thm1}
Let $\eta_1, \ldots, \eta_t$ be positive real algebraic numbers in a real algebraic number field $\mathbb{K} \subset \mathbb{R}$ of degree $D$. Let $b_1, \ldots, b_t$ be nonzero integers such that 
\begin{align*}
\Gamma := \eta_1 ^{b_1} \ldots \eta_t ^{b_t} - 1 \neq 0.
\end{align*}
Then
\begin{align*}
\log | \Gamma| > - 1.4 \times 30^{t+3} \times t^{4.5} \times D^2 (1 + \log D)(1 + \log B)A_1 \ldots A_t,
\end{align*}
where
\begin{align*}
B \geq \max \{|b_1|, \ldots, |b_t| \},
\end{align*}
and
\begin{align*}
A_j \geq \max \{ D h(\eta_j), |\log \eta_j|, 0.16  \}, \quad \text{for all} \quad j = 1, \ldots,t.
\end{align*}
\end{theorem}

\subsection{Reduction procedure.}
The bounds on the variables obtained via Baker's theorem are usually too large for any computational purposes. In order to get further refinements, we use the Baker-Davenport reduction procedure. The variant we apply here is the one due to Dujella and Peth\H{o} (\cite{DP}, Lemma 5a). For a real number $r$, we denote by $\parallel r \parallel$ the quantity $\min \{|r - n| : n \in \mathbb{Z} \}$, the distance from $r$ to the nearest integer.
\begin{lemma}[Dujella, Peth\H o, \cite{DP}]\label{DP}
Let $\kappa \neq 0, A, B$ and $\mu$ be real numbers such that $A > 0$ and $B > 1$. Let $M > 1$ be a positive integer and suppose that $\frac{p}{q}$ is a convergent of the continued fraction expansion of $\kappa$ with $q > 6M$. Let 
\begin{align*}
\varepsilon := \parallel \mu q \parallel - M \parallel \kappa q \parallel.
\end{align*}
If $\varepsilon > 0$, then there is no solution of the inequality
\begin{align*}
0 < |m \kappa - n + \mu| < AB^{-k}
\end{align*}
in positive integers $m,n,k$ with
\begin{align*}
\frac{\log (Aq/\varepsilon)}{\log B} \leq k \quad \text{and} \quad m \leq M.
\end{align*}
\end{lemma}

\noindent
Lemma \ref{DP} cannot be applied when $\mu=0$ (since then $\varepsilon<0$). In this case, we use the following criterion due to Legendre, a well--known result from the theory of Diophantine approximation. For further details, we refer the reader to the books of Cohen \cite{HC1, HC2}.
\begin{lemma}[Legendre, \cite{HC1, HC2}]
\label{lg}
Let $\kappa$ be real number and $x,y$ integers such that
\begin{align*}
\left|\kappa-\frac{x}{y}\right|<\frac{1}{2y^2}.
\end{align*}
Then $x/y=p_k/q_k$ is a convergent of $\kappa$. Furthermore, let  $ M $ and $ N $ be a nonnegative integers such that $ q_N> M $. Then putting $ a(M):=\max\{a_{i}: i=0, 1, 2, \ldots, N\} $, the inequality 
\begin{align*}
\left|\kappa-\frac{x}{y}\right|\ge \frac{1}{(a(M)+2)y^2},
\end{align*}
holds for all pairs $ (x,y) $ of positive integers with $ 0<y<M $.
\end{lemma}

\noindent
In some cases, it is necessary to apply the LLL-algorithm. This algorithm is used to find an effective lower bound for $ |\lambda_1 x_1+\lambda_2 x_2| $, where $ \lambda_1, \lambda_2 $ are real numberd, and $ x_2,x_2 $ are rational integers such that $ |x_j|\le X_j \in \mathbb{N} $ for $ j=1,2 $. We state the following lemma, which is the direct consequence of Cohen (\cite{HC1}, Proposition 2.3.20).
\begin{lemma}\label{LLL}
Keeping the above assumptions, let $ C $ be a fixed (large) positive constant.  Let $ \Lambda $ be the lattice generated by the columns of the matrix
\begin{align*}
\begin{pmatrix}
1 & 0 \\ 
\lfloor C\lambda_1 \rceil & \lfloor C\lambda_2 \rceil
\end{pmatrix} .
\end{align*}
Consider a reduced basis $ \{\mathbf{v_1}, \mathbf{v_2}\} $ of $ \Lambda$ together with $ \{\mathbf{v_1^*}, \mathbf{v_2^*}\}$ as its Gram-Schmidt associated basis. Put
\begin{align*}
d_\Lambda= \dfrac{\Vert \mathbf{v_1} \Vert}{\max\{1,\Vert \mathbf{v_1} \Vert/\Vert \mathbf{v_2^*} \Vert \}} \quad \text{and} \quad T=\dfrac{1+X_1^2+X_2^2}{2}.
\end{align*}
If $ d_\Lambda^2\ge T^2 +X_1^2 $, the we have that
\begin{align*}
 |\lambda_1 x_1+\lambda_2 x_2| \ge \dfrac{\sqrt{d_\Lambda^2-X_1^2}-T}{C}.
\end{align*}
\end{lemma}

\noindent We will also need the following lemma by G\'{u}zman S\'{a}nchez and Luca (\cite{GSL}, Lemma 7):
\begin{lemma}[G\'{u}zman S\'{a}nchez, Luca, \cite{GSL}]\label{l3}
Let $r \geq 1$ and $H > 0$ be such that $H > (4r^2)^r$ and $H > L/(\log L)^r$. Then
\begin{align*}
L < 2^r H (\log H)^r.
\end{align*}
\end{lemma}

\section{Proof of the Main Result.}
\subsection{The low range, $ n\in [0,300] $.}
A simple computer program in Sage was used to check all the solutions of the Diophantine equation \eqref{eq1} for the parameters $ d_1\neq d_2 \in \{0, \ldots, 9\},~ d_1>0 $ and $ 1\le \ell,m \le n\le 500 $ The only solutions found are those stated in Theorem \ref{thm1x}. From now onwards, we assume that $n > 500$.

\subsection{The initial bound on $n$.}
We note that \eqref{eq1} can be rewritten as

\begin{align}\label{eq3}
S_n & = \overline{\underbrace{d_1\cdots d_1}_\text{$\ell$ times} \underbrace{d_2\cdots d_2}_\text{$m$ times} \underbrace{d_1\cdots d_1}_\text{$\ell$ times}} \nonumber \\
& = \overline{\underbrace{d_1\cdots d_1}_\text{$\ell$ times}} \times 10^{\ell + m} + \overline{ \underbrace{d_2\cdots d_2}_\text{$m$ times}} \times 10^\ell + \overline{\underbrace{d_1\cdots d_1}_\text{$\ell$ times}} \nonumber \\
\Rightarrow ~S_n & = \frac{1}{9} \left(d_1 \times 10^{2\ell + m} - (d_1 - d_2) \times 10^{\ell+m} + (d_1-d_2) \times 10^\ell -  d_1 \right).
\end{align}

\noindent The next lemma relates the sizes of $n$ and $2\ell + m$.

\begin{lemma}\label{l4}
All solutions of the Diophantine equation \eqref{eq3} satisfy the inequality below

\begin{equation*}
(2\ell + m) \log 10 - 3 < n \log \alpha < (2\ell + m) \log 10 + 1.
\end{equation*}

\end{lemma}

\begin{proof}
\noindent From Lemma  \ref{lem1}, we have $\alpha^{n-1} \le  S_n <  \alpha^{n+1}$. So, we note that 
\begin{equation*}
\alpha^{n-1} \le  S_n < 10^{2\ell + m}.
\end{equation*}

\noindent Taking the logarithm on both sides of the preceeding inequality and simplifying,  we get 

\begin{equation*}
n \log \alpha < (2\ell + m) \log 10 + \log \alpha.
\end{equation*}

\noindent Hence, using the estimates in \eqref{Trib4}, we get
\begin{equation}\label{qw1}
n \log \alpha < (2\ell + m) \log 10 + 1. 
\end{equation}
\noindent
On the other hand, using Lemma \ref{lem1},  we also note that 
\begin{align*}
10^{2\ell + m -1} < S_n <  \alpha^{n+1}.
\end{align*}
Taking logarithms on both sides and simplifying as before gives
\begin{align*}
(2\ell +m)\log 10-\log 10 - \log \alpha < n\log \alpha.
\end{align*}
Similary, using the estimates in \eqref{Trib4}, we get that
\begin{equation}\label{qw2}
(2\ell +m)\log 10-3 < n\log \alpha.
\end{equation}
Clearly, from \eqref{qw1} and \eqref{qw2}, we get the desired inequality in Lemma \ref{l4}, as required.
\end{proof}

\noindent We now proceed to examine the Diophantine equation \eqref{eq3} in three different steps as follows.

\noindent \textbf{Step 1.} From the equations \eqref{Trib3} and \eqref{eq3}, we have that

\begin{equation*}
9(\alpha^n +  \beta^n +  \gamma^n )= d_1 \times 10^{2\ell + m} - (d_1-d_2)\times 10^{m+\ell} + (d_1-d_2)\times 10^\ell - d_1.
\end{equation*}

\noindent Hence,

\begin{equation*}
9 \alpha^n  - d_1 \times 10^{2\ell + m} = -9\xi(n) - (d_1-d_2)\times 10^{m+\ell} + (d_1-d_2)\times 10^\ell - d_1.
\end{equation*}

\noindent Thus, we have that

\begin{align*}
|9 \alpha^n  - d_1 \times 10^{2\ell + m}| & = |-9\xi(n) - (d_1-d_2)\times 10^{m+\ell} + (d_1-d_2)\times 10^\ell - d_1| \\
& \leq 9 \alpha^{-n/2} + 27 \times 10^{m + \ell}  < 28 \times 10^{m + \ell},
\end{align*}

\noindent where we used the fact that $n > 500$. Dividing both sides by $d_1 \times 10^{2\ell + m}$, we get

\begin{equation}\label{eq4}
\left | \left( \frac{9}{d_1} \right) \alpha^n \times 10^{-2\ell -m} - 1  \right | < \frac{28 \times 10^{m+\ell}}{d_1 \times 10^{2\ell +m}} \le  \frac{28}{10^\ell}.
\end{equation}

\noindent We put

\begin{equation}
\Gamma_1 := \left( \frac{9}{d_1} \right) \alpha^n \times 10^{-2\ell -m} - 1 .
\end{equation}

\noindent We shall proceed to compare this upper bound on $|\Gamma_1|$ with the lower bound we deduce from Theorem \ref{thm1}. Note that $\Gamma_1 \neq 0$, otherwise this would imply that $$ \alpha^n = \frac{10^{2\ell + m} \times d_1}{9}.$$ Then,  applying the Galois automorphism, $\sigma$ on both sides of the preceeding equation and taking absolute values, we get that

\begin{equation*}
\left| \frac{10^{2\ell + m} \times d_1}{9} \right| =\left|\sigma\left( \frac{10^{2\ell + m} \times d_1}{9}\right) \right|= |\sigma ( \alpha^n)| = |\beta^n| < 1,
\end{equation*}

\noindent which is a contradiction. Thus, we have that $\Gamma_1 \neq 0$.\\

\noindent With a view towards applying Theorem \ref{thm1}, we define the following parameters:

\begin{equation*}
\eta_1 := \frac{9}{d_1}, \ \eta_2 := \alpha, \ \eta_3 := 10, \ b_1 := 1, \ b_2 := n, \ b_3 := -2\ell - m, \ t: = 3.
\end{equation*}

\noindent Since $$10^{2\ell +m -1} < S_n <  \alpha^{n+1}<10^{n-1},$$ we have that $2\ell + m < n$. Thus, we take $B = n$. We also note that $\mathbb{K} = \mathbb{Q}(\eta_1, \eta_2, \eta_3) = \mathbb{Q}(\alpha)$. Hence, $D = [\mathbb{K}: \mathbb{Q}] =  [\mathbb{Q}(\alpha) \mathbb{Q}] = \deg (\Psi) = 3$.
\noindent Furthermore, using the properties of $ h $ in \eqref{heights},  we note that

\begin{equation*}
h(\eta_1) = h \left( \frac{9}{d_1} \right) \leq h(9) + h(d_1) \leq 2\log 9=4\log 3.
\end{equation*}

\noindent We also have that $h(\eta_2) = h(\alpha) = \frac{1}{3}\log \alpha$ and $h(\eta_3) = \log 10$. Hence, we let

\begin{equation*}
A_1 := 12\log 3, \ A_2 := \log \alpha, \ A_3 := 3 \log 10.
\end{equation*}

\noindent Thus, we deduce via Theorem \ref{thm1} that

\begin{align*}
\log |\Gamma_1| & > -1.4 \cdot 30^6 \cdot  3^{4.5} \cdot  3^2 (1 + \log 3)(1 + \log n)(12\log 3)(\log \alpha)(3 \log 10) \\&> -7.17\times 10^{13}(1 + \log n).
\end{align*}

\noindent Comparing the last inequality obtained above with \eqref{eq4}, we get

\begin{equation*}
\ell \log 10 - \log 28 < 7.17 \times 10^{13}(1 + \log n).
\end{equation*}

\noindent Hence,

\begin{equation}\label{eq5}
\ell \log 10 <  7.18 \times 10^{13}(1 + \log n).
\end{equation}

\noindent \textbf{Step 2.} We rewrite equation \eqref{eq3} as

\begin{equation*}
9 \alpha^n  - d_1 \times 10^{2\ell + m} + (d_1-d_2)\times 10^{\ell+m} = -9\xi(n)  + (d_1-d_2)\times 10^\ell - d_1.
\end{equation*}

\noindent That is, 

\begin{equation*}
9 \alpha^n  - (d_1 \times 10^\ell - (d_1-d_2))\times 10^{\ell+m} = -9\xi(n)  + (d_1-d_2)\times 10^\ell - d_1.
\end{equation*}

\noindent Hence, we get

\begin{align*}
\left|9 \alpha^n  - (d_1 \times 10^\ell - (d_1-d_2))\times 10^{\ell+m}\right| & = \left|-9\xi(n)  + (d_1-d_2)\times 10^\ell - d_1\right|\\
& \leq \frac{9}{\alpha^{n/2}} + 18 \times 10^\ell < 19 \times 10^\ell .
\end{align*}

\noindent Dividing throughout by $(d_1 \times 10^\ell - (d_1-d_2))\times 10^{\ell+m}$, we get that

\begin{align}\label{eq6}
\left|\left(\frac{9a}{d_1 \times 10^\ell - (d_1-d_2)} \right) \alpha^n  \times 10^{-\ell - m}  -1 \right|  < \frac{19 \times 10^\ell}{(d_1 \times 10^\ell - (d_1-d_2))\times 10^{\ell+m}}   < \frac{19}{10^m}.
\end{align}

\noindent We put

\begin{equation*}
\Gamma_2: = \left(\frac{9a}{d_1 \times 10^\ell - (d_1-d_2)} \right) \alpha^n \times 10^{-\ell -m} - 1.
\end{equation*}

\noindent Apply the Galois automorphism $ \sigma $ before, we have that $\Gamma_2 \neq 0$. Otherwise,  this would imply that 

\begin{equation*}
 \alpha^n = 10^{\ell + m}  \left(\frac{d_1 \times 10^\ell - (d_1-d_2)}{9} \right),
\end{equation*}

\noindent which in turn implies that

\begin{equation*}
\left| 10^{\ell + m}  \left(\frac{d_1 \times 10^\ell - (d_1-d_2)}{9} \right) \right| = |\sigma( \alpha^n)| = | \beta^n | < 1,
\end{equation*}

\noindent which again leads to a contradiction. In preparation towards applying Theorem \ref{thm1}, we define the following parameters:

\begin{equation*}
\eta_1: = \frac{9}{d_1 \times 10^\ell - (d_1-d_2)} , \ \eta_2: = \alpha, \ \eta_3: = 10, \ b_1: = 1, \ b_2 := n, \ b_3 := -\ell - m, \ t := 3.
\end{equation*}

\noindent In order to determine what $A_1$ will be, we need to find the find the maximum of the quantities $h(\eta_1)$ and $|\log \eta_1|$. We note that
\begin{align*}
h(\eta_1) & = h \left(\frac{9}{d_1 \times 10^\ell - (d_1-d_2)} \right)\\
& \leq  h(9) + \ell h(10) +h(d_1) + h(d_1-d_2) + \log 2\\
& \leq 10 \log 3 + \ell \log 10\\
& < 10\log 3 + 7.18\times 10^{13} (1+ \log n)\\
& < 7.19\times 10^{13}(1 + \log n),
\end{align*}

\noindent where, in the second last inequality above, we used the bound given in \eqref{eq5}. 

\noindent
On the other hand, we also have that
\begin{align*}
|\log \eta_1| & = \left| \log \left(\frac{9}{d_1 \times 10^\ell - (d_1-d_2)} \right) \right| \\
& \leq  \log 9 + |\log (d_1 \times 10^\ell - (d_1-d_2))|\\
& \leq  \log 9 + \log (d_1 \times 10^\ell) + \left| \log \left( 1 - \frac{d_1 - d_2}{d_1 \times 10^\ell}  \right)  \right|\\
& \leq \ell \log 10 + \log d_1 + \log 9  + \frac{|d_1 - d_2|}{d_1 \times 10^\ell} + \frac{1}{2} \left(  \frac{|d_1 - d_2|}{d_1 \times 10^\ell} \right)^2 + \cdots \\
& \leq \ell \log 10 + 2 \log 9 + \frac{1}{10^\ell} + \frac{1}{2 \times 10^{2\ell}} + \cdots\\
& \leq 7.18 \times 10^{13}(1 + \log n) + 4 \log 3 + \frac{1}{10^\ell - 1} \\
& < 7.20 \times 10^{13}(1 + \log n),
\end{align*}

\noindent where, in the second last inequality, we also used the bound given in \eqref{eq5}. Clearly,  we observe that $D h(\eta_1) > |\log \eta_1|$.

\noindent Thus,  let $A_1 := 21.51 \times 10^{13} (1 + \log n)$. We take $A_2: = \log \alpha$ and $A_3 := 3 \log 10$, as defined in \textbf{Step 1}. Similarly, as before we take $B := n$.  Theorem \ref{thm1} then tells us that

\begin{align*}
\log |\Gamma_2| & > -1.4 \cdot  30^6 \cdot 3^{4.5} \cdot 3^2 \cdot (1 + \log 3)(1 + \log n)(\log \alpha)(3 \log 10)A_1\\
& > -5.44 \cdot 10^{12}(1 + \log n)A_1 > -1.55 \times 10^{28}(1 + \log n)^2.
\end{align*}

\noindent Comparing the last inequality with (\ref{eq6}), we have that

\begin{align}\label{eq7}
m \log 10 & < 1.55 \times 10^{28}(1 + \log n)^2 + \log 19 < 1.56 \times 10^{28}(1+\log n)^2.
\end{align}

\noindent \textbf{Step 3.} We rewrite equation (\ref{eq3}) as 

\begin{equation*}
9 \alpha^n  - d_1 \times 10^{2\ell + m} + (d_1-d_2)\times 10^{\ell+m} - (d_1-d_2)\times 10^\ell = -9\xi(n) - d_1.
\end{equation*}

\noindent Therefore, we have

\begin{align*}
\left| 9\alpha^n - \left(d_1 \times 10^{\ell+m} + (d_1-d_2)\times 10^m - (d_1-d_2)\right)\times 10^\ell  \right| & = |-9\xi(n) - d_1|\\
& \leq \frac{9}{\alpha^{n/2}} + 9 < 10.
\end{align*}

\noindent Hence,

\begin{equation}\label{eq8}
\left| \frac{1}{9} \left(d_1 \times 10^{\ell+m} + (d_1-d_2)\times 10^m - (d_1-d_2)\right)\times \alpha^{-n} \times 10^\ell - 1 \right| < \frac{10}{9 \alpha^n} < \frac{2}{\alpha^n}.
\end{equation}

\noindent Let

\begin{equation*}
\Gamma_3 :=  \frac{1}{9}\left (d_1 \times 10^{\ell+m} + (d_1-d_2)\times 10^m - (d_1-d_2)\right)  \times \alpha^{-n} \times 10^\ell - 1.
\end{equation*}

\noindent As before, we have that $\Gamma_3 \ne 0$. Otherwise, we would have that

\begin{equation*}
 \alpha^n = \frac{1}{9} \left(d_1 \times 10^{\ell+m} + (d_1-d_2)\times 10^m - (d_1-d_2)\right) \times 10^\ell.
\end{equation*}

\noindent Applying the automorphism $\sigma$, of the Galois group $\mathcal{G}$ on both sides of the above equation and then taking absolute values, we get that

\begin{equation*}
\left| \frac{1}{9} \left(d_1 \times 10^{\ell+m} + (d_1-d_2)\times 10^m - (d_1-d_2)\right) \times 10^\ell \right| = |\sigma ( \alpha^n)| = | \beta^n| < 1,
\end{equation*}

\noindent which is a contradiction. We would now like to apply Theorem (\ref{thm1}) to $\Gamma_3$. To this end, we let:

\begin{align*}
\eta_1 &:=   \frac{1}{9}\left( d_1 \times 10^{\ell+m} + (d_1-d_2)\times 10^m - (d_1-d_2) \right), \\ \eta_2 & := \alpha, \ \eta_3 := 10, \ b_1 := 1, \ b_2: = -n, \ b_3: = \ell, \ t: = 3.
\end{align*}

\noindent As in the previous cases, we can take $B: = n$ and $D: = 3$. Using the properties of $ h $ given in \eqref{heights},  we note that
\begin{align*}
h(\eta_1) & \leq h(9) + h(d_1) + (\ell+m)h(10) + h(d_1 - d_2) + m h(10) + h(d_1 - d_2) + 3 \log 2\\
& \leq 16\log 3 + (2m+\ell) \log 10.
\end{align*}

\noindent Using equations (\ref{eq5}) and (\ref{eq7}), we have that
\begin{align}\label{eq9}
(2m + \ell)\log 10 & < 3.12\times 10^{28} (1 + \log n)^2 + 7.18 \times 10^{13}(1 + \log n) < 3.13 \times 10^{28}(1 + \log n)^2.
\end{align}

\noindent Thus, we conclude that
\begin{equation*}
h(\eta_1) < 3.14 \times 10^{28}(1 + \log n)^2.
\end{equation*}

\noindent We now find an upper bound for $|\log \eta_1|$. We have that

\begin{align*}
|\log \eta_1| & = \left| \log \left(  \frac{1}{9} \left(d_1 \times 10^{\ell+m} + (d_1-d_2)\times 10^m - (d_1-d_2) \right)\right) \right| \\
& \le \log 9  + \left|\log \left(d_1 \times 10^{\ell+m} + (d_1-d_2)\times 10^m - (d_1-d_2)\right)\right|\\
& \leq 3 \log 3 + \log (d_1 \times 10^{\ell + m}) + \left| \log \left(  1 - \frac{(d_1 - d_2)(10^m - 1)}{d_1 \times 10^{\ell+m}}  \right)  \right|\\
& \le 5 \log 3 + (\ell+m) \log 10 + \left| \log \left(  1 - \frac{(d_1 - d_2)(10^m - 1)}{d_1 \times 10^{\ell+m}}  \right)  \right|\\
& \le 5 \log 3 + (\ell+m) \log 10 + \frac{|(d_1 - d_2)(10^m - 1)|}{d_1 \times 10^{\ell + m}} + \frac{1}{2} \left( \frac{|(d_1 - d_2)(10^m - 1)|}{d_1 \times 10^{\ell + m}} \right)^2 + \cdots \\
& \le 5 \log 3 + \ell\log 10 +m\log 10 + \frac{1}{10^l} + \frac{1}{2 \times 10^{2\ell}} + \cdots\\
& \le 5 \log 3 + 7.18\times 10^{13}(1+\log n)+ 1.56\times 10^{28}(1+\log n)^2 +\frac{1}{10^\ell - 1}\\
& < 1.58 \times 10^{28}(1 + \log n)^2,
\end{align*}

\noindent where, in the last inequality above, we used the bound from (\ref{eq9}). We note that $D \cdot h(\eta_1) > |\log \eta_1|$. Thus, we let $A_1 = 9.42 \times 10^{28}(1 + \log n)^2$, $A_2 = \log \alpha$ and $3 \log 10$. Theorem \ref{thm1} then implies that

\begin{align*}
\log |\Gamma_3| & > -5.44 \times 10^{12} (1 + \log n)A_1> -5.13 \times 10^{42} (1 + \log n)^3.
\end{align*}

\noindent Comparing the last inequality with \eqref{eq8}, we deduce that

\begin{equation*}
n \log \alpha < 5.13 \times 10^{42} (1 + \log n)^3 + \log 3.
\end{equation*}

\noindent Therefore, we get that 

\begin{equation*}
n < 8.52 \times 10^{42} (\log n)^3.
\end{equation*}

\noindent With the notation of Lemma \ref{l3}, let $r: = 3$, $L := n$ and $H: = 8.52 \times 10^{42}$ and notice that this data meets the conditions of the lemma. Lemma \ref{l3}, tells us that

\begin{equation*}
n < 2^3 \times 8.52 \times 10^{42} (\log (8.52 \times 10^{42}))^3.
\end{equation*}

\noindent After a simplification, we obtain the  bound
\begin{equation*}
n < 6.6 \times 10^{50}.
\end{equation*}

\noindent Lemma \ref{l4} then implies that

\begin{equation*}
2\ell +m < 1.8 \times 10^{50}.
\end{equation*}

\noindent The following lemma summarizes what we have proved up to this far:

\begin{lemma}\label{lemx}
All solutions to the Diophantine equation \eqref{eq1} satisfy the following inequalities
\begin{equation*}
2\ell + m < 1.8 \times 10^{50} \quad \text{and} \quad n < 6.6 \times 10^{50}.
\end{equation*}

\end{lemma}

\subsection{Reduction of the bounds.}
The bounds obtained in Lemma \ref{lemx} are too large to carry out meaningful computations with the computer. Thus, we need to reduce them. To do so, we apply Lemma \ref{DP} several times as follows.

\noindent
First, we return to the inequality \eqref{eq4} and put
\begin{align*}
z_1:=(2\ell+m)\log 10 -n\log\alpha +\log \left(\frac{d_1}{9}\right).
\end{align*}
The inequality \eqref{eq4} can be rewritten as
\begin{align*}
|\Gamma_1|=\left|e^{-z_1}-1\right| < \dfrac{28}{10^{\ell}}.
\end{align*}
We assume that $ \ell\ge 2 $, then the right--hand side of the above inequality is at most $ 28/100 < 1/2 $. The inequality $ |e^z-1|< x $ for real values of $ x $ and $ z $ implies that $ z<2x $. Thus,
\begin{align*}
|z_1|< \dfrac{56}{10^{\ell}}.
\end{align*}
This implies that
\begin{align*}
\left|(2\ell+m)\log 10 -n\log\alpha -\log \left(\frac{9}{d_1}\right)\right| < \dfrac{56}{10^{l}}.
\end{align*}
Dividing through the above inequality by $ \log\alpha $ gives
\begin{align}\label{kala}
\left|(2\ell+m)\frac{\log 10}{\log\alpha} - n + \frac{\log(d_1/9)}{\log\alpha}\right|< \dfrac{56}{10^{\ell}\log\alpha}.
\end{align}
So, we apply Lemma \ref{DP} with the quantities:
\begin{align*}
\kappa:=\dfrac{\log 10}{\log \alpha}, \quad \mu(d_1):=\frac{\log(d_1/9)}{\log\alpha}, \quad 1\le d_1\le 9, \quad A:=\dfrac{56}{\log\alpha}, \quad B:=10.
\end{align*}
Let $$ \kappa=[a_0; a_1, a_2, \ldots]=[3; 1, 3, 1, 1, 14, 1, 3, 3, 6, 1, 13, 3, 4, 2, 1, 1, 2, 3, 3, 2, 2, 1, 2, 5, 1, 1,39, 2, 1, \ldots], $$ be the continued fraction expansion of $ \kappa $. We set $ M:=10^{51} $ which is the upper bound on $ 2\ell+m $. With the help of a simple program in SageMath, we find out that the convergent
\begin{align*}
\dfrac{p}{p}=\dfrac{p_{98}}{q_{98}}=\dfrac{39444948689252707738489528190760067813905266021850462}{10439083718875559984715310681234336679649552673602845},
\end{align*}
is such that $ q=q_{98}>6M $. Furthermore, it gives $ \varepsilon > 0.00227519 $, and thus,
\begin{align*}
\ell\le \dfrac{\log((56/\log\alpha)q/\varepsilon)}{\log 10}<56.
\end{align*}
For the case $ d_1=9 $, we have that $ \mu(d_1)=0 $. In this case, it is not possible to reduce the bound via Lemma \ref{DP}, so we apply Lemma \ref{lg}. The inequality \eqref{kala} can be rewritten as
\begin{align*}
\left|\frac{\log 10}{\log \alpha} - \frac{n}{2\ell+m}  \right| < \frac{56}{10^\ell(2\ell+m) \log \alpha}<\dfrac{1}{2(2\ell+m)^2},
\end{align*}
because $ 2\ell+m<10^{51}:=M $. It follows from Lemma \ref{lg} that $ \frac{n}{2\ell+m} $  is a convergent of $ \kappa:=\frac{\log 10}{\log\alpha} $. So $ \frac{n}{2\ell+m} $ is of the form $ p_k/q_k $ for some $ k=0, 1, 2, \ldots, 98 $. Thus,
\begin{align}\label{leg}
\dfrac{1}{(a(M)+2)(2\ell+m)^2}\le \left|\dfrac{\log 10}{\log \alpha}-\frac{n}{2\ell+m} \right|<\frac{56}{10^\ell(2\ell+m) \log \alpha}.
\end{align}
With the aid of a simple computer program in SageMath, we have $$ a(M)=\max\{a_{k}: k=0, 1, 2, \ldots, 98\}=44. $$ Thus, using \eqref{leg} we get that
\begin{align*}
\ell \leq \dfrac{\log\left(\dfrac{46\times 56 \times 10^{51}}{\log\alpha}\right)}{\log 10} < 55.
\end{align*}
So, $ \ell\le 56 $ in both cases. In the case $ \ell<2 $, we have that $ \ell<2<56 $. Therefore, $ \ell\le 56 $ holds in all cases.

\noindent
Next, for fixed $ d_1\neq d_2\in\{0, \ldots, 9\}, ~d_1>0 $ and $ 1\le \ell\le 56 $, we return to the inequality \eqref{eq6} and put
\begin{align*}
z_2:=(\ell+m)\log 10 - n\log\alpha +\log\left(\frac{d_1\times 10^{\ell}-(d_1-d_2)}{9a}\right).
\end{align*}
From the inequality \eqref{eq6}, we have that
\begin{align*}
|\Gamma_2| = \left|e^{-z_2}-1\right|< \dfrac{19}{10^{m}}.
\end{align*}
Assume that $ m\ge 2 $, then the right--hand side of the above inequality is at most $ 19/100 < 1/2 $. Thus, we have that
\begin{align*}
|z_2|<  \dfrac{38}{10^{m}},
\end{align*}
which implies that
\begin{align}\label{kala2}
\left|(\ell+m)\log 10 - n\log\alpha +\log\left(\frac{d_1\times 10^{\ell}-(d_1-d_2)}{9}\right)\right| <  \dfrac{38}{10^{m}}.
\end{align}
Dividing through by $ \log\alpha $ gives
\begin{align*}
\left|(\ell+m)\frac{\log 10}{\log\alpha} - n +\log\left(\frac{(d_1\times 10^{\ell}-(d_1-d_2))/9a}{\log\alpha}\right)\right| <  \dfrac{38}{10^{m}\log\alpha}.
\end{align*}
Thus, we apply Lemma \ref{DP} with the quantities:
\begin{align*}
 \quad \mu(d_1,d_2):=\log\left(\frac{(d_1\times 10^{\ell}-(d_1-d_2))/9}{\log\alpha}\right), \quad A:=\dfrac{38}{\log\alpha}, \quad B:=10.
\end{align*}
We take the same $ \kappa $ and its convergent $ p/q =p_{98}/q_{98} $ as before. Since $ \ell+m < 2\ell+m $, we set $M :=10^{51} $ as the upper bound on $ \ell+m $. With the help of a simple computer program in SageMath, for $ \ell\in [1,56] $,  we get that $ \varepsilon >0.0000604124 $, and therefore,
\begin{align*}
m\le \dfrac{\log((38/\log\alpha)q/\varepsilon)}{\log 10}<58.
\end{align*}
For the case $ (d_1, d_2, \ell)=(1,0,1) $, we have that $ \mu(d_1,d_2)=0 $. In this case, it is also not possible to reduce the bound via Lemma \ref{DP}, so we apply Lemma \ref{LLL}. The inequality \eqref{kala2} can be rewritten as
\begin{align}\label{kala3}
\log \left|(\ell+m)\log 10 - n\log\alpha \right| <  \log 38 - m\log 10.
\end{align}
Next we reduce the upper bound for $ m $ as follows. Let $ \lambda_1=\log 10 $, $ \lambda_2=\log\alpha $, $ x_1=\ell+m $, $ x_2=-n $, and a sufficiently large $ C= 10^{110} $. The LLL algorithm uses the lattice $ \Lambda $ generated by the columns of the matrix
\begin{align*}
\begin{pmatrix}
1&0\\ \lfloor   10^{110}\log 10\rceil & \lfloor  10^{110} \log \alpha \rceil
\end{pmatrix}.
\end{align*}
A simple computer program in SageMath is used to compute the  reduced basis $ \{\mathbf{v}_1, \mathbf{v}_2\} $ for $ \Lambda $. 
Then, the Gram-Schmidt associated basis $ \{\mathbf{v}_1^{*}, \mathbf{v}_2^{*}\} $ is obtained using the standard procedure:
\begin{align*}
\mathbf{v}_1^{*}=\mathbf{v}_1 \quad \text{and} \quad \mathbf{v}_2^{*}=\mathbf{v}_2-\dfrac{\langle \mathbf{v}_2, \mathbf{v}_1\rangle}{\langle \mathbf{v}_1, \mathbf{v}_1 \rangle}\mathbf{v}_1.
\end{align*}
Let $ X_1=1.8\times 10^{50} $ and $ X_2=6.6\times 10^{50} $. Then, it is clear that $ |x_i|\le X_i $ for $ i=1,2 $. Now, we calculate 
\begin{align*}
d_\Lambda= \dfrac{\Vert \mathbf{v_1} \Vert}{\max\{1,\Vert \mathbf{v_1} \Vert/\Vert \mathbf{v_2^*} \Vert \}}=5.10\times 10^{109} \quad \text{and} \quad T=\dfrac{1+X_1^2+X_2^2}{2}=2.34\times 10^{101}.
\end{align*}
Since $ d_\Lambda^2=2.6\times 10^{219}>5.48\times 10^{202}=T^2+X_1^2 $, 
Lemma \ref{LLL} guarantees that 
\begin{align}\label{kala4}
\log \left|(\ell+m)\log 10 - n\log\alpha \right| > \log(0.0000512). 
\end{align}
Comparing \eqref{kala3} with \eqref{kala4}, we get that $ m\le 6 $. Thus, we have that $ m\le 58 $ in both cases. The case $ m<2 $ holds as well since $ m<2<58 $.

\noindent
Lastly, for fixed $ d_1\neq d_2\in\{0, \ldots, 9\}, ~d_1>0 $, $ 1\le \ell\le 56 $ and $ 1\le m\le 58 $, we return to the inequality \eqref{eq8} and put
\begin{align*}
z_3:=\ell\log 10 -n\log\alpha +\log\left(\dfrac{d_1\times 10^{\ell+m}+(d_1-d_2)\times 10^{m}-(d_1-d_2)}{9}\right).
\end{align*}
From the inequality \eqref{eq8}, we have that
\begin{align*}
|\Gamma_3|=\left|e^{z_3}-1\right|< \frac{2}{\alpha^{n}}.
\end{align*}
Since $ n>500 $, the right--hand side of the above inequality is less than $ 1/2 $. Thus, the above inequality implies that
\begin{align*}
|z_3|<  \frac{4}{\alpha^{n}},
\end{align*}
which leads to
\begin{align*}
\left|\ell\log 10 -n\log\alpha +\log\left(\dfrac{d_1\times 10^{\ell+m}+(d_1-d_2)\times 10^{m}-(d_1-d_2)}{9}\right)\right| <  \frac{4}{\alpha^{n}}.
\end{align*}
Dividing through by $ \log\alpha $ gives,
\begin{align*}
\left|\ell\frac{\log 10}{\log\alpha} -n +\log\left(\dfrac{(d_1\times 10^{\ell+m}+(d_1-d_2)\times 10^{m}-(d_1-d_2))/9}{\log\alpha}\right)\right| <  \frac{4}{\alpha^{n}\log\alpha}.
\end{align*}
Again, we apply Lemma \ref{DP} with the following data:
\begin{align*}
 \quad \mu(d_1, d_2):=\log\left(\dfrac{(d_1\times 10^{\ell+m}+(d_1-d_2)\times 10^{m}-(d_1-d_2))/9}{\log\alpha}\right), \quad A:=\dfrac{4}{\log \alpha}, \quad B:=\alpha.
\end{align*}
We take the same $ \kappa $ and its convergent $ p/q=p_{98}/q_{98} $ as before. Since $ \ell<2\ell+m $, we choose $ M:=10^{51} $ as the upper bound for $ \ell $. With the help of a simple computer program in SageMath, we get that $ \varepsilon >0.000000106965 $, and thus,
\begin{align*}
n\le \dfrac{\log((4/\log\alpha)q/\varepsilon)}{\log\alpha}<226.
\end{align*}
Thus, we have that $ n\le 226 $. This contradicts with our initial assumption that $ n> 500 $. Hence, Theorem \ref{thm1x} is completely proved. \qed

\section*{Acknowledgements.}
This research was done when the author visited the Max Planck Institute for Mathematics in Bonn for the period 1 August to 31 October 2025. He greatly thanks this institution for the hospitality and a fruitful working environment.

\section*{Declarations}
\subsection*{Data Availability Statement}
All the data generated or analysed during this study are included in this article.

\subsection*{Conflict of interest} The author declares no conflict of competing interests.

\section*{Addresses}
\noindent
$ ^{1} $ Department of Mathematics, Makerere University, Kampala, Uganda

\noindent
$ ^{2} $ Max Planck Institute for Mathematics, Bonn, Germany

\noindent
Email: \url{mahadi.ddamulira@mak.ac.ug}


\begin{thebibliography}{99}

\bibitem{BMS}
Y Bugeaud, M Mignotte, and S Siksek,
\textit{Classical and modular approaches to exponential Diophantine equations I. Fibonacci and Lucas perfect powers},
Annals of Mathematics, (2), {\bf 163}(2):969--1018, 2006.

\bibitem{HC1}
H. Cohen, {\it A Course in Computational Algebraic Number Theory}, Springer, New York, 1993.

\bibitem{HC2}
H. Cohen, {\it Number Theory. Volume I: Tools and Diophantine Equations}, Springer, New York, 2007.


\bibitem{MD}
T P Chalebgwa and M Ddamulira
\textit{Padovan numbers which are palindromic concatenations of two distinct repdigits}, Rev. R. Acad. Cienc. Exactas F\'is. Nat. Ser. A Mat. RACSAM, {\bf 115}(3):Art. No. 108, 14 pp, 2021.


\bibitem{MD1}
M Ddamulira,
\textit{Padovan numbers that are concatenations of two distinct repdigits}, Math. Slovaca, {\bf 71}(2):275--284, 2021.

\bibitem{MD2}
M Ddamulira,
\textit{Tribonacci numbers that are concatenations of two repdigits},
Rev. R. Acad. Cienc. Exactas F\'is. Nat. Ser. A Mat. RACSAM, {\bf 114}(4):Art. No. 203, 10 pp, 2020.

\bibitem{MD3}
M Ddamulira, P. Emong, and G. I. Mirumbe
\textit{Palindromic concatenations of two distinct repdigits in Narayana's cows sequence}, Bull. Iranian Math. Soc., {\bf 50}(3):Art. No. 35, 16pp, 2024.




\bibitem{DP}
A Dujella and A Peth\H{o},
\textit{A generalization of a theorem of Baker and Davenport}, The
Quarterly Journal of Mathematics, (2), {\bf 49}(195):291--306, 1998.


\bibitem{GSL}
S G\'{u}zman S\'{a}nchez and F Luca,
\textit{Linear combinations of factorials and $s$-units in a binary recurrence sequence},
Annales Math\'{e}matiques du Qu\'{e}bec, {\bf 38}(2):169--188, 2014.





\bibitem{OEIS}
OEIS Foundation Inc. (2022). {\it The On-Line Encyclopedia of Integer Sequences}, \url{http://oeis.org/A001644}. 



\end{thebibliography}
\end{document}